\documentclass[12pt]{amsart}
\usepackage{amsmath}
\usepackage{amsthm}
\usepackage{amssymb}
\usepackage{verbatim}
\usepackage{subfigure}
\usepackage{verbatim}
\usepackage{url}
\numberwithin{equation}{section}
\usepackage{textcomp}
\usepackage{stmaryrd}
\usepackage{xy}
\usepackage{color}
\usepackage[pdftex]{graphicx}

\definecolor{Mygrey}{gray}{0.75}

\def\displayandname#1{\rlap{$\displaystyle\csname #1\endcsname$}%
                      \qquad \texttt{\char92 #1}}

\makeatletter
\def\url@leostyle{%
  \@ifundefined{selectfont}{\def\UrlFont{\sf}}{\def\UrlFont{\small\ttfamily}}}
\makeatother
\newtheorem{thm}{Theorem}[section]
\newtheorem{pro}[thm]{Proposition}

\newtheorem{cla}[thm]{Claim}

\newtheorem{cor}[thm]{Corollary}

\theoremstyle{definition}
\newtheorem{df}[thm]{Definition}

\theoremstyle{remark}
\newtheorem{rem}[thm]{Remark}
\newtheorem{exa}[thm]{Example}

\begin{document}

\bibliographystyle{acm}

\title{Characteristic Polynomials of Supertropical Matrices}
  \author[Adi Niv]{ Adi Niv$^\dagger$ }

  \thanks{$^\dagger$ Department of Mathematics, Bar-Ilan University, Ramat Gan 52900, Israel.\newline
Email:  {\tt adi.niv@live.biu.ac.il}} 
\thanks{This paper is part of the author's Ph.D thesis, which was written at Bar-Ilan University under the supervision of Prof. L.\ H.\ Rowen.}

\begin{abstract} Supertropical matrix theory was investigated in ~\cite{STMA3}, whose terminology we follow. In this work we investigate eigenvalues, characteristic polynomials and coefficients of characteristic polynomials of supertropical matrices and their powers, and obtain the analog to the basic property of matrices  that any power of an eigenvalue of a matrix is an eigenvalue of the corresponding power of the matrix. 
\end{abstract}

\maketitle

\thispagestyle{myheadings}
\font\rms=cmr8
\font\its=cmti8
\font\bfs=cmbx8

\markright{}
\def\thepage{}

\section{Introduction}
The theory of supertropical matrices, laid out in ~\cite{STMA}, shows that although the ghost ideal in supertropical semifield  $R=T\bigcup G\bigcup \{-\infty \}$, where~$G$ is a copy of~$T$ obtained as the image of the one-to-one ghost map~$\nu:T\rightarrow G$ sending~$a\mapsto \nu(a)=a^{\nu}$,  helps to recover many of the classical matrix theory properties, supertropical matrices still show  a very undesired behavior when one  wants to use them to decompose the supertropical vector-space $R^n$ as the sum of the eigenspaces of a matrix.
 Previous work on the connection between supertropical matrices and their powers shows that the pathological behavior of a matrix, such as ghost level of its entries, its singularity, and dependency of its eigenvectors, can be avoided by passing to high enough powers of the matrix (see ~\cite{TA}). In this paper we obtain a direct result (Theorem 3.6), that for any matrix $A$, the characteristic polynomial of $A^m$ \textgravedbl ghost surpasses\textacutedbl the characteristic polynomial of $A$, for any natural $m$, and thus equality holds when they are tangible.

\vskip 0.25 truecm
\noindent We establish some fundamental definitions and properties for our work.

\begin{df}
A \textbf{track of a permutation $\pi\in S_n$} is the sequence $$a_{1,\pi(1)}a_{2,\pi(2)}\cdots a_{n,\pi(n)}$$ of $n$ entries of a matrix $A=(a_{i,j})\in M_n(R)$. We denote the identity permutation track, which lies on the main diagonal of a matrix, as the $Id$ track, and the permutation track which lies on the secondary diagonal of a matrix, that is~$a_{1,n}a_{2,n-1}\cdots a_{n-1,2}a_{n,1}$, as the~$-Id$ track.
\end{df}

\vskip 0.25 truecm

\begin{df}
We define the \textbf{tropical determinant} of a tangible matrix $A=(a_{i,j})$ to be the usual permanent

    $$|A| =\sum_{\sigma \in S_n}a_{1\sigma(1)}\cdots a_{n\sigma(n)}$$ (see ~\cite[\S 5]{MPA}).  We refer to the permutation track yielding the highest value in this sum as \textbf{the dominant permutation track}.
\end{df}

\vskip15pt

\vskip 0.25 truecm
We notice that over the supertropical structure the determinant is tangible whenever only one tangible permutation track dominates this sum, and ghost whenever either a ghost permutation track dominates it or at least two permutation tracks dominate it.
\vskip 0.25 truecm

\begin{thm} \textbf{(The rule of determinants)}
For $n\times n$ matrices A,B over the supertropical division semiring R, we have $|AB|=|A||B| + g,\ \ where\ g\in G$. In particular if $|AB|$ is tangible then $|AB|=|A||B|$.
\end{thm}

\noindent{\textbf{Proof.}} See Theorem 3.5 in ~\cite[\S 3]{STMA}.

\vskip 0.5 truecm

This property can easily be understood by observing the determinant function in classical linear algebra. We know that the result in this case would be~$|AB|=|A||B|$ because every monomial in $|AB|$ that does not appear in $|A||B|$  will appear twice with opposite signs. The meaning of this result in supertropical linear algebra is that every non-$|A||B|$ monomial appears twice, and therefore will create a ghost element. If a monomial from $|A||B|$  dominates then $|AB|=|A||B|$, and $|AB|\in G$ otherwise.

\vskip 0.5 truecm
\pagenumbering{arabic}
\addtocounter{page}{1}
\markboth{\SMALL ADI NIV }{\SMALL CHARACTERISTIC POLYNOMIALS OF SUPERTROPICAL MATRICES}

\begin{rem} \textbf{(The Frobenius property)}

\noindent a.	$\forall n \in$ $\mathbb{N}$,   $y \in R \quad \quad (x+y)^n=x^n + y^n +ghost$.

\noindent b.	If R is a supertropical algebra then $(x+y)^n=x^n + y^n$.

\end{rem}

\noindent{\textbf{Proof.}} See ~\cite[Remark 1.3]{STMA}.

\begin{df}
We say that $b$ is a $k$ \textbf{root} of $a$, for some $k\in N$, denoted as~$b= \sqrt[k]{a}$ if~$b^k=a$.
\end{df}

\begin{rem}
If $a,b\in T$ and $a^k=b^k$ then $a^k+b^k=(a+b)^k\in G$, therefore $a+b\in G$ and $a=b$. That is, the $k$ root of a tangible element is unique.
 \end{rem}

\section{Supertropical polynomials and the relation $\models_{gs}$.}

As one can see in ~\cite{STA}, the polynomials over the supertropical structure are rather simple to understand geometrically. Viewing a monomial $a_ix^i\in R[x]$ with the operations max-plus, it is easy to notice that the display of such a monomial is a line represented in the classical operations plus-multiplication as $ix+a_i$, and therefore the power $i$ of $x$ represents the slope of the monomial. Since $T$ is ordered we may present its elements on an axis, directed rightward, where if $a<b$ then $a$ appears left to $b$ on the $T$-axis, for every pair of distinct elements $a,b\in T$. It is now easy to understand that a supertropical polynomial $$\sum_{i=0}^n a_ix^i\in R[x]$$ takes the value of the dominant monomial among $a_ix^i$ along the $T$-axis. That having been said, it is possible that some monomials in the polynomial would not dominate for any $x\in R$ (see ~\cite[Definition 4.9]{STA}).

\begin{df}
Let $$f(x)=\sum_{i=0}^n a_ix^i\in R[x]$$ be a supertropical polynomial. We call monomials in $f(x)$ that dominate for some~$x\in R$ \textbf{essential}, and monomials in $f(x)$ that do not dominate for any $x\in R$ \textbf{inessential}. We write $$f^{es}=   \sum_{i\in I} a_ix^i\in R[x],$$ where $a_ix^i\ \ \forall i\in I$ is an essential monomial, called \textbf{the essential polynomial} of $f$.
 \end{df}

\begin{df}
We define an element $a \in R$ the \textbf{root} of a polynomial $f(x)$ if $$f(a)\in G\bigcup \{0_R\}$$.
\end{df}

We distinguish between roots of a polynomial that are obtained as common values of two leading tangible monomials and roots obtained as a value of a leading ghost monomial.

\begin{df} We refer to roots of a polynomial being obtained as an intersection of two leading tangible monomials as \textbf{corner roots}, and to  roots that are being obtained from one leading  ghost monomial as \textbf{non-corner roots}.
\end{df}

\vskip 0.25 truecm

\begin{rem} Suppose $f\in R[x]$.

\noindent 1. The constant term $a_0$ and the leading monomial $a_nx^n$ will dominate first and last, respectively, due to their slopes. Furthermore, they are the only ones that are \textbf{necessarily} essential in every polynomial.

\noindent 2. For the intersection between an essential monomial $a_ix^i$ and the next essential monomial $a_jx^j$ where $j>i$, we notice that the multiplicity of the root $\alpha _i$ is $k=j-i$,  because $a_ix^i$ dominates all lines between $a_ix^i$ and $a_jx^j$, or in other words $\frac{a_i}{a_j}x^i=b_ix^i$ dominates all lines between $\frac{a_i}{a_j}x^i=b_ix^i$ and $x^j$, and we get:

\begin{center}
    \noindent{$(x^j+b_{j-1}x^{j-1}+...+b _ix^i)=(x^j+b _ix^i)=x^i(x^k+b _i)$}

\noindent{$=x^i(x^k+(\sqrt[k]{b _i})^k)=x^i(x+\sqrt[k]{b _i})^k=x^i(x+\alpha _i)^k$.}
\end{center}

\end{rem}

\vskip 0.25 truecm

\begin{df}
We denote a polynomial $f$ as \textbf{a-primary} if $a$ is the only corner root of $f$.
\end{df}

\begin{cla}
For any $a$-primary polynomial $f\in R[x]$, $a=\sqrt[n]{\frac{ a_0}{a_n}}$.
\end{cla}

\begin{proof}
As a result of Remark 2.4. we know that the first and last monomials of the polynomial have to dominate all other monomials for the first and last segments respectively; therefore if there is only one corner root then it has to be the common value of the first monomial $a_0$ and the last monomial $a_nx^n$.
\end{proof}

\begin{df}
Let $a,b$ be any two elements in $R$. We say that $a$ \textbf{ghost surpasses} $b$, denoted $a\models_{gs} b$, if~$a=b+ghost$, i.e. $a=b\ $ or $a\in G$ with $a^{\nu}\geq b^{\nu}$.

For matrices $A=(a_{ij}),B=(b_{ij})\in M_{n\times m}(R)$ (and in particular for vectors) $A\models_{gs} B$ means $a_{ij}\models_{gs} b_{ij}\quad \forall i=1,...,n$ and $j=1,...,m$.

For polynomials~$f(x)=\sum a_ix^i,\ g(x)=\sum b_ix^i\in R_n[x]$, we say that $f(x)\models_{gs} g(x)$ when~$a_{i}\models_{gs}b_{i}\ \forall i=0,...,n$.
\end{df}

\begin{center}
    \textbf{Important properties of $\models_{gs}$:}
\end{center}

\noindent{1. $\models_{gs}$ is an order relation.}

See ~\cite[Lemma 1.5]{STMA2}.

\vskip 0.25 truecm

\noindent{2. If $a\models_{gs} b$ then $ac\models_{gs} bc$.}

\vskip 0.25 truecm

\section{Supertropical characteristic polynomials and eigenvalues.}

 We follow the description as studied in ~\cite[\S 5]{STMA2},

\begin{df}
$\forall v \in T^n \quad \mbox{and} \quad A \in M_n(T)$ such that $\exists x \in T $ such that $ Av\models_{gs} xv$ we say that $v$ is a \textbf{supertropical eigenvector}  of $A$ with a \textbf{supertropical eigenvalue}~$x$.

\noindent The \textbf{characteristic polynomial} of A is set to be $f_A(x)=|xI+A|$, and the tangible value of its roots are the eigenvalues of $A$ as shown in ~\cite[Theorem 7.10]{STMA}.
\end{df}

\begin{pro}
If $ x \in T$ is an eigenvalue of a matrix $A \in M_n(T)$, then $x^i$ is an eigenvalue of  $A^i$.
\end{pro}

\begin{proof}

\noindent $ x \in T$ is an eigenvalue of the matrix A so $\exists v \in T^n$ such that $Av\models_{gs}xv $.

\noindent If $i=2$:
 $$A^2v=AAv\models_{gs}xAv\models_{gs}xxv=x^2v.$$

\noindent We assume that the claim holds for $i-1$ and prove it for general $i$:
$$ A^iv=A^{i-1}Av \models_{gs}xA^{i-1}v\models_{gs}x\cdot x^{i-1}v=x^iv.$$
\end{proof}

\vskip 0.5 truecm

However, we notice that $\{x^i:x\quad \mbox{is an eigenvalue of a given matrix A}\}$ need not be the only eigenvalues of the matrix $A^i$, as shown in the next example.

\begin{exa}

Consider the $2\times 2$ matrix

$$A=
\left(
\begin{array}{cc}
0 & 0\\
1 & 2
\end{array}
\right).$$

\vskip 0.25 truecm

\noindent{Then $f_A(x)=x^2+2x+2 \Rightarrow f_A(x) \in G\quad \mbox{when}\quad x=0,2$.}

\vskip 0.25 truecm

\noindent{However,}

\vskip 0.25 truecm

$$A^2=
\left(
\begin{array}{cc}
1 & 2\\
3 & 4
\end{array}
\right).$$

\vskip 0.25 truecm

\noindent Which means $f_{A^2}(x)=x^2+4x+5^{\nu} \Rightarrow f_{A^2}(x) \in G \quad \mbox{when}\quad x^{\nu}\leq 1^{\nu} \quad \mbox{or}\quad x=4$.
\end{exa}

\vskip 0.25 truecm

\begin{rem}
We notice that the coefficient of $x^i$ in the characteristic polynomial is the sum of all determinants of $(n-i)\times (n-i)$ minors obtained by erasing $i$ rows and their corresponding columns. For example the coefficient of $x^{n-1}$ would be the sum of all determinants of $1\times 1$ minors obtained by erasing $n-1$ rows and their corresponding columns, yielding the trace. And the coefficient of $x^{0}$ would be the determinant.
\end{rem}

\begin{cla}
For every monomial $a_{i_1,\pi(i_1)} a_{i_2,\pi(i_2)}\cdots a_{i_k,\pi(i_k)}$  in the determinant of a $k\times k$ minor in~A (and therefore in the coefficient of $x^{n-k}$ in $f_{A}$) the term $$(a_{i_1,\pi(i_1)} a_{i_2,\pi(i_2)}\cdots a_{i_k,\pi(i_k)})^m$$ will appear exactly once in the coefficient of $x^{n-k}$ in $f_{A^m}$ since it will appear in (and only in) the permutation track of $\pi^m$ of the corresponding $k\times k$ minor in $A^m$.
\end{cla}

\begin{proof}

$\forall a_{i_j,\pi(i_j)}$ in the permutation track of $\pi \in S_k$ in every $k\times k$ minor in $A$, we look at the~$i_j,\pi ^m(i_j)$ position in $A^m$: $$\sum_{t_{j,l}=1}^n a_{i_j,t_{j,1}} a_{t_{j,1},t_{j,2}}\cdots a_{t_{j,m-1},\pi ^m(i_j)}.$$ In this sum there exists a monomial of the form $a_{i_j,\pi (i_j)}a_{\pi (i_j),\pi ^2(i_j)}\cdots a_{\pi ^{m-1}(i_j),\pi ^m(i_j)}$; therefore  in the permutation track of $\pi ^m$  in the corresponding $k\times k$ minor in $A^m$ there exists the monomial:

\begin{equation}[a_{i_1,\pi (i_1)}\ a_{\pi (i_1),\pi ^2(i_1)}\ \cdots\ a_{\pi ^{m-1}(i_1),\pi ^m(i_1)}]\cdot\end{equation}

\begin{center}
$\cdot [a_{i_2,\pi (i_2)}\ a_{\pi (i_2),\pi ^2(i_2)}\ \cdots\ a_{\pi ^{m-1}(i_2),\pi ^m(i_2)}]\cdot$          

 $\vdots$ 

$\ \cdot [a_{i_k,\pi (i_k)}\ a_{\pi (i_k),\pi ^2(i_k)}\ \cdots\ a_{\pi ^{m-1}(i_k),\pi ^m(i_k)}].$ 
\end{center}
\vskip 0.25 truecm

 Since $\pi \in S_k$, every $i_j$ has exactly one source and one image both belong to $\{i_1,...,i_k\}$, therefore, looking at the columns in 3.1 as rows:

\begin{equation}[a_{i_1,\pi (i_1)}\ \ \ \ a_{i_2,\pi (i_2)}\ \ \ \ \cdots\ \ \ \   a_{i_k,\pi (i_k)}]\cdot\end{equation}

\begin{center}   
$\cdot [a_{\pi (i_1),\pi ^2(i_1)}\ \ \ a_{\pi (i_2),\pi ^2(i_2)}\ \cdots\ \ \ a_{\pi (i_k),\pi ^2(i_k)}]\cdot$

 $\vdots$ 

$\cdot [a_{\pi ^{m-1}(i_1),\pi ^m(i_1)} a_{\pi ^{m-1}(i_2),\pi ^m(i_2)}\cdots a_{\pi ^{m-1}(i_k),\pi ^m(i_k)}]$
\end{center}

\vskip 0.25 truecm

\noindent and then rearranging each row, we obtain:
   
\begin{equation}[a_{i_1,\pi (i_1)}a_{i_2,\pi (i_2)}\cdots a_{i_k,\pi (i_k)}]\cdot\end{equation}
\begin{center}
$\cdot [a_{i_1,\pi (i_1)}a_{i_2,\pi (i_2)}\cdots a_{i_k,\pi (i_k)}]\cdot$

 $ \vdots$ 

$\cdot [a_{i_1,\pi (i_1)}a_{i_2,\pi (i_2)}\cdots a_{i_k,\pi (i_k)}]$

\end{center}

\noindent which means each $a_{i_j,\pi (i_j)}$ will appear exactly $m$ times in (3.1) yielding the monomial $$(a_{i_1,\pi(i_1)} a_{i_2,\pi(i_2)}\cdots a_{i_k,\pi(i_k)})^m$$ where $a_{i_1,\pi(i_1)} a_{i_2,\pi(i_2)}\cdots a_{i_k,\pi(i_k)}$ is in the monomial of the permutation track of $\pi$  in the corresponding~$k\times k$ minor in $A$. 

This monomial will not appear again in the coefficient of $x^{n-k}$ since the only entries in this monomial are of the form $a_{\pi ^l(i_j),\pi ^{l+1}(i_j)}$ and can only be followed by $a_{\pi ^{l+1}(i_j),\pi ^{l+2}(i_j)}$, for every $l=0,...,m,\ \ j=1,...,k$, in this monomial, which yields only the  monomials  $$a_{i_j,\pi (i_j)}a_{\pi (i_j),\pi ^2(i_j)}\cdots a_{\pi ^{m-1}(i_j),\pi ^m(i_j)}$$ of length $m$, which appear only on the permutation track of $\pi ^m$ of the unique $k\times k$-minor that includes $i_1,...,i_k$. 
\end{proof}

\begin{thm}
Let $f_{A^m}(x)=\sum \beta _ix^i$ be the characteristic polynomial of the $m^{th}$ power of an $n\times n$ matrix $A$, and $f_{A}(x)=\sum \alpha_ix^i$ be the characteristic polynomial of $A$. Then~$f_{A^m}(x^m) \vDash f_A(x)^m$, which means $$\beta _i\vDash \alpha _i^m \ \ \forall i=1,...,n.$$
\end{thm}

\begin{proof}

\noindent Looking at the monomials in the $(i_j,\sigma(i_j))$ position of $A^m$: $$a_{i_j,t_{j,1}} a_{t_{j,1},t_{j,2}}\cdots a_{t_{j,m-1},\sigma(i_j)}$$ as tracks from $i_j$ to $t_{j,1}$ to $t_{j,2}$ to ... to $\sigma (i_j)$, we may say that  $$\sum_{t_{j,l}=1}^n a_{i_j,t_{j,1}} a_{t_{j,1},t_{j,2}}\cdots a_{t_{j,m-1},\sigma (i_j)}$$ describes all the ways to get from  $i_j$ to $\sigma(i_j)\in S_k$ in $m$ steps. Therefore a permutation track of~$\sigma\in S_k$ in a $k\times k$ minor in $A^m$ describes all the ways to obtain a permutation track from~$i_1$ to $\sigma (i_1)$, from $i_2$ to $\sigma (i_2)$,..., from $i_k$ to $\sigma (i_k)$. 

According to Claim 3.5, for every monomial $a_{i_1,\pi(i_1)} a_{i_2,\pi(i_2)}\cdots a_{i_k,\pi(i_k)}$  from a $k\times k$ minor in~$A$,~$(a_{i_1,\pi(i_1)} a_{i_2,\pi(i_2)}\cdots a_{i_k,\pi(i_k)})^m$ will appear in the $\pi^m$ permutation track in the corresponding minor in $A^m$. So we already have obtained $b_i=a_i^m+g,\ g\in R$ where~$k=n-i$. It remains to show that $g\in G$.

\noindent If
 \begin{equation}[a_{i_1,t_{1,1}} a_{t_{1,1},t_{1,2}}\cdots a_{t_{1,m-1},\sigma (i_1)}]   [a_{i_2,t_{2,1}} a_{t_{2,1},t_{2,2}}\cdots a_{t_{2,m-1},\sigma (i_2)}]  \cdots   [a_{i_k,t_{k,1}} a_{t_{k,1},t_{k,2}}\cdots a_{t_{k,m-1},\sigma (i_k)}]\end{equation}

 \noindent from a $k\times k$-minor in~$A^m$ is not of the form~$(a_{i_1,\pi(i_1)}a_{i_2,\pi(i_2)}\cdots a_{i_k,\pi(i_k)})^m$ for some~$\pi \in S_k$ such that~$\pi ^m=\sigma$, then we factor the permutation $\sigma$ into disjoint cycles:
$$(j_1 , \sigma (j_1), \sigma ^2(j_1),..., \sigma ^{k_1}(j_1))\circ(j_2, \sigma (j_2), \sigma ^2(j_2),..., \sigma ^{k_2}(j_2))\circ...\circ(j_d , \sigma (j_d), \sigma ^2(j_d),..., \sigma ^{k_d}(j_d)),$$
\noindent where
\footnotesize $$\{j_1 , \sigma (j_1), \sigma ^2(j_1),..., \sigma ^{k_1}(j_1)\}\bigcup\{j_2, \sigma (j_2), \sigma ^2(j_2),..., \sigma ^{k_2}(j_2)\}\bigcup...\bigcup\{j_d , \sigma (j_d), \sigma ^2(j_d),..., \sigma ^{k_d}(j_d)\}$$ \normalsize coincides with $\{i_1,...,i_k\} $ and $\sigma ^{t_1} (j_l)\ne \sigma ^{t_2}(j_l)$ where ${t_1}\ne {t_2}$, for every $l=1,...,d$, unless~${t_1}=0, {t_2}=k_l+1$, and then $\sigma^{k_l+1}(j_l)=j_l$. 

\vskip 0.5 truecm

We rewrite (3.4) as:  
\begin{equation}([a_{ j_1,t_{1_{1,1}}} a_{t_{1_{1,1}},t_{1_{1,2}}}\cdots a_{t_{1_{1,m-1}},\sigma (j_1)}] \cdots[a_{\sigma ^{k_1}(j_1),t_{1_{k_1,1}}} a_{t_{1_{k_1,1}},t_{1_{k_1,2}}}\cdots a_{t_{1_{k_1,m-1}}, \sigma ^{(k_1+1)}(j_1)}])\cdot\end{equation}
\begin{center}
 $\cdot ([a_{j_2,t_{2_{1,1}}} a_{t_{2_{1,1}},t_{2_{1,2}}}\cdots a_{t_{2_{1,m-1}},\sigma (j_2)}]   \cdots [a_{\sigma ^{k_2}(j_2),t_{2_{k_2,1}}} a_{t_{2_{k_2,1}},t_{2_{k_2,2}}}\cdots a_{t_{2_{k_2,m-1}}, \sigma ^{(k_2+1)}(j_2)}])\cdot$

$\vdots$
 
$\cdot([a_{ j_d,t_{d_{1,1}}}\ \cdots \ a_{t_{d_{1,m-1}},\sigma (j_d)}] \cdots [a_{\sigma ^{k_d}(j_d),t_{d_{k_d+1,1}}} a_{t_{d_{k_d+1,1}},t_{d_{k_d+1,2}}}\ \cdots \  a_{t_{d_{k_d+1,m-1}}, \sigma ^{(k_d+1)}(j_d)}])$
\end{center}

\vskip 0.25 truecm

\noindent where $[\  ]$ denotes entries from $A^m$ and $(\  )$ denotes the beginning and end of a cycle, a notation we will use throughout the whole proof. 

For convenience we write 3.5 as:
$$\prod_{r=1}^d ([a_{ j_r,t_{r_{1,1}}} a_{t_{r_{1,1}},t_{r_{1,2}}}\cdots a_{t_{r_{1,m-1}},\sigma (j_r)}]\cdots   [a_{\sigma ^{k_r}(j_r),t_{r_{k_r,1}}} a_{t_{r_{k_r,1}},t_{r_{k_r,2}}}\cdots a_{t_{r_{k_r,m-1}}, \sigma ^{(k_r+1)}(j_r)}])$$

\noindent and we would want to show that there exists another permutation track (also in a $k\times k$ minor) where this monomial will appear again. 
\vskip 0.25 truecm

In order to do so we observe the first index $t_{r_{i,1}}$  which appears in the second element of each entry from $A^m$:

$$a_{t_{r_{i,1}},t_{r_{i,2}}}\ \ \forall r=1,...,d,\ i=1,...,k_r.$$
\vskip 0.25 truecm

 \noindent \underline{Case I:} All of these indices are different. Then we may begin every cycle with the second element and end it with the original first element:

$$\prod_{r=1}^d ([a_{t_{r_{1,1}},t_{r_{1,2}}}\cdots a_{t_{r_{1,m-1}},\sigma (j_r)} a_{ \sigma (j_r),t_{r_{2,1}}} ]   [a_{t_{r_{2,1}},t_{r_{2,2}}}\cdots a_{t_{r_{2,m-1}},\sigma ^2(j_r)} a_{ \sigma (j_r),t_{r_{3,1}}}]\cdots$$
$$\cdots [a_{t_{r_{k_r,1}},t_{r_{k_r,2}}}\cdots a_{t_{r_{k_r,m-1}}, \sigma ^{(k_r+1)}(j_r)} a_{ j_r,t_{r_{1,1}}}]).$$
\vskip 0.25 truecm

\noindent The fact that there are $k$ different indices means that these new disjoint cycles would describe a permutation (with the same cycle sizes as $\sigma$), which graphically may be represented as: 
\vskip 0.25 truecm

\includegraphics{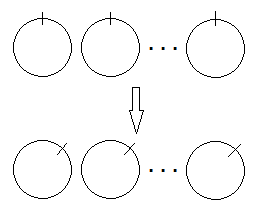}

\noindent For example: $([a_{1,1}a_{1,2}][a_{2,2}a_{2,1}])=([a_{1,2}a_{2,2}][a_{2,1}a_{1,1}])$. 
\vskip 0.15 truecm

 \noindent \underline{Case II:} Some index $t_{r_i,1}$ repeats in the same cycle. Now we may factor this cycle into two sub-cycles and not change any other cycle:

$$( [a_{ j_r,t_{r_{1,1}}} a_ {\underline{\underline{t_{r_{1,1}}}},t_{r_{1,2}}}\cdots a_{t_{r_{1,m-1}},\sigma (j_r)}][a_{ \sigma (j_r),t_{r_{2,1}}} a_{t_{r_{2,1}},t_{r_{2,2}}}\cdots a_{t_{r_{2,m-1}},\sigma ^2(j_r)}]\cdots$$ $$\cdots[a_{ \sigma ^{i-1} (j_r),t_{r_{1,1}}} a_{\underline{ t_{r_{1,1}}},t_{r_{i,2}}}\cdots a_{t_{r_{i,m-1}},\sigma ^i (j_r)}][a_{\sigma ^{k_r}(j_r),t_{r_{k_r,1}}} a_{t_{r_{k_r,1}},t_{r_{k_r,2}}}\cdots a_{t_{r_{k_r,m-1}}, \sigma ^{(k_r+1)}(j_r)}])=$$
$$( [a_{ j_r,t_{r_{1,1}}} a_{\underline{ t_{r_{1,1}}},t_{r_{i,2}}}\cdots a_{t_{r_{i,m-1}},\sigma ^i (j_r)}]\cdots[a_{\sigma ^{k_r}(j_r),t_{r_{k_r,1}}} a_{t_{r_{k_r,1}},t_{r_{k_r,2}}}\cdots a_{t_{r_{k_r,m-1}},\sigma ^{(k_r+1)}(j_r)}])\cdot$$ $$ ([a_{ \sigma (j_r),t_{r_{2,1}}} a_{t_{r_{2,1}},t_{r_{2,2}}}\cdots a_{t_{r_{2,m-1}},\sigma ^2(j_r)}]\cdots[a_{ \sigma ^{i-1} (j_r),t_{r_{1,1}}} a_{\underline{\underline{t_{r_{1,1}}}},t_{r_{1,2}}}\cdots a_{t_{r_{1,m-1}},\sigma (j_r)}])$$

\noindent yielding a new permutation, which graphically is represented as: 

\includegraphics{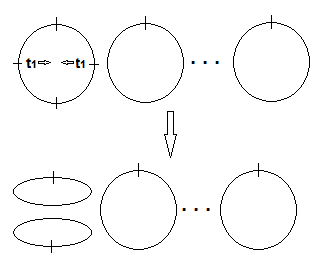}

\noindent For example:  $([a_{1,2}a_{2,2}][a_{2,2}a_{2,1}])=([a_{1,2}a_{2,1}])([a_{2,2}a_{2,2}])$. 
\vskip 0.25 truecm

 \noindent \underline{Case III:}  Some index $t_{r_i,1}$ repeats in two different cycles. Now we may join these cycles into a single cycle and not change any other cycle:

$$([a_{ j_r,t_{r_{1,1}}} a_{\underline{\underline{ t_{r_{1,1}}}},t_{r_{1,2}}}\cdots a_ { t_{r_{1,m-1}},\sigma (j_r)}]\cdots  [a_{\sigma ^{k_r}(j_r),t_{r_{k_r,1}}} \cdots a_{t_{r_{k_r,m-1}}, \sigma ^{(k_r+1)}(j_r)}])\cdot$$
$$([a_{ j_v,t_{v_{1,1}}}\cdots a_{t_{v_{1,m-1}},\sigma (j_v)}]\cdots [a_{ \sigma ^{i-1} (j_v),t_{r_{1,1}}} a_{\underline{ t_{r_{1,1}}},t_{v_{i,2}}}\cdots a_{t_{v_{i,m-1}},\sigma ^i (j_v)}] \cdots  [a_{\sigma ^{k_v}(j_v),t_{v_{k_v,1}}} \cdots a_{t_{v_{k_v,m-1}}, \sigma ^{(k_v+1)}(j_v)}])=$$
$$([a_{ j_r,t_{r_{1,1}}} a_{\underline{ t_{r_{1,1}}},t_{v_{i,2}}}\cdots a_{t_{v_{i,m-1}},\sigma ^i (j_v)}] \cdots  [a_{\sigma ^{k_v}(j_v),t_{v_{k_v,1}}} \cdots  a_{t_{v_{k_v,m-1}}, \sigma ^{(k_v+1)}(j_v)}][a_{ j_v,t_{v_{1,1}}} \cdots a_{t_{v_{1,m-1}},\sigma (j_v)}]\cdots$$ $$\cdots [a_{ \sigma ^{i-1} (j_v),t_{r_{1,1}}} a_{\underline{\underline{ t_{r_{1,1}}}},t_{r_{1,2}}}\cdots a_ { t_{r_{1,m-1}},\sigma (j_r)}]    \cdots  [a_{\sigma ^{k_r}(j_r),t_{r_{k_r,1}}} \cdots a_{t_{r_{k_r,m-1}}, \sigma ^{(k_r+1)}(j_r)}])$$

\vskip 0.25 truecm

\noindent yielding a new permutation, which graphically is represented as: 

\includegraphics{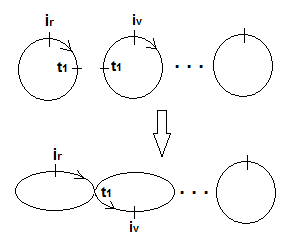}

\vskip 0.5 truecm

For example:   $([a_{1,2}a_{2,2}][a_{2,2}a_{2,1}])([a_{3,2}a_{2,4}][a_{4,2}a_{2,3}])=([a_{1,2}a_{2,4}][a_{4,2}a_{2,3}][a_{3,2}a_{2,2}][a_{2,2}a_{2,1}])$.

\vskip 0.5 truecm

To summarize, let~$f_{A^m}(x)=\Sigma b_ix^i$ and~$f_{A}(x)=\Sigma a_ix^i$. Then, $a_i^m$  appears as a monomial in~$b_i$ for every~$i=0,...,n$, and every other monomial in~$b_i$ will appear twice. Thus~$b_i\vDash a_i^m\ \ \forall i=0,...,n$ and therefore~$f_{A^m}(x^m)\vDash f_A(x)^m$.
\end{proof}

\begin{cor}
If the coefficients of $f_{A^m}$ are tangible then $f_{A^m}(x^m)=f_A(x)^m$.
\end{cor}

\begin{exa} $ $

\noindent  1. Looking at the constant coefficients $\alpha _0,\ \beta _0$, which represent the determinants of $A,\ A^m$ respectively, this theorem provides an alternative proof of Theorem~1.3 in the special case that $B$ is a power of the matrix $A$. 

$ $

\noindent 2. Let $A=(a_{ij})$ be an $n\times n$ matrix. For any $m\in \mathbb{N}$ we have $tr(A^m)\vDash (tr(A))^m$.

\noindent Here is a short, self-contained proof.

\noindent Every diagonal entry $b_{i,i}$ in $$A^m=(b_{i,j})=(\sum_{k_t=1}^n a_{i,k_1}a_{k_1,k_2}\cdots a_{k_{m-1},j})$$ is a sum of monomials such that one of the monomials (where $k_t=i,\ \ \forall t$) is $(a_{i,i})^m$ and for every other monomial $(a_{i,k_1}a_{k_1,k_2}\cdots a_{k_{m-1},i})$ there exists $t\in \{1,...,n\}$ such that~$k_t\ne i$. Hence there exists an equal monomial $(a_{k_t,k_{t+1}}\cdots a_{k_{m-1},i}a_{i,k_1}\cdots a_{k_{t-1},k_{t}})$ in the $(k_t,k_t)$ position of $A^m$. Therefore  $$tr(A^m)=\Sigma_{i=1}^n b_{i,i}=\Sigma_{i=1}^n (a_{i,i})^m+ghost=(tr(A))^m+ghost.$$ 

\end{exa}

\vskip 0.5 truecm

\begin{exa}
Let $$A=(a_{i,j})\in M_3(R).$$

 \noindent Therefore, $f_A(x)=x^3+tr(A)x^2+\alpha _1x+detA$. We already saw the  property $\beta _i\vDash \alpha _i^m$ for the trace and the determinant. We look at the coefficient of $x$, $\alpha_1$, which is the sum of all determinants of the~$\ 3-1\times 3-1=2\times 2$ minors obtained by erasing one row and its corresponding column:
$$\alpha _1=\underbrace{(\underbrace{a_{1,1}a_{2,2}}_\text{Id track}+\underbrace{a_{1,2}a_{2,1}}_\text{-Id track})}_\text{$1^{st}\ 2\times 2\ minor$}+
\underbrace{(\underbrace{a_{1,1}a_{3,3}}_\text{Id track}+\underbrace{a_{1,3}a_{3,1}}_\text{-Id track})}_\text{$2^{nd}\ 2\times 2\ minor$}+
\underbrace{(\underbrace{a_{2,2}a_{3,3}}_\text{Id track}+\underbrace{a_{2,3}a_{3,2}}_\text{-Id track})}_\text{$3^{rd}\ 2\times 2\ minor$}$$.

\vskip 0.5 truecm

Let us consider
$$A^2=(\sum_{t=1}^3 a_{i,t}a_{t,j})$$
 We look at the corresponding coefficient $\beta _1$ in $f_{A^2}$ which is obtained from  three $2\times 2$ minors, each contributing two permutation tracks as described:

\vskip 0.5 truecm

\noindent $-Id$ tracks in $A^2$  of the form $\sum( a_{i,t_1}a_{t_1,j})(a_{j,t_2}a_{t_2,i}),\ i\ne j$,

\noindent $Id$ tracks  in $A^2$  of the form $\sum (a_{i,t_1}a_{t_1,i})(a_{j,t_2}a_{t_2,j}),\ i\ne j$.

\vskip 0.5 truecm

\begin{center} The monomials originated from $\alpha _1$:\end{center}

 Every Id permutation track $a_{i,i}a_{j,j}$ from $\alpha _1$ must correspond to $Id^2=Id$ permutation track in $\beta _1$. Indeed, for  $ (a_{i,t_1}a_{t_1,i})(a_{j,t_2}a_{t_2,j})$ where $t_1=i$ and $t_2=j$ we obtain $(a_{i,i}a_{j,j})^2$.

 In that same way, every -Id permutation track $a_{i,j}a_{j,i}$ from $\alpha _1$ must correspond to~$(-Id)^2=Id$ permutation track in $\beta _1$. Indeed, for~$( a_{i,t_1}a_{t_1,i})(a_{j,t_2}a_{t_2,j})$ where~$t_1=j$ and~$t_2=i$ we obtain~$(a_{i,j}a_{j,i})^2$.

\vskip 0.5 truecm

\begin{center}All other monomials appear twice:\end{center}

For every $[(a_{i,t_1}a_{t_1,j})(a_{j,t_2}a_{t_2,i})]$ in the permutation tracks of $-Id$ there exists another -Id permutation track $ [(a_{t_1,j}a_{j,t_2})(a_{t_2,i}a_{i,t_1})]$  where $t_1\ne t_2$, and equals $[(a_{j,t_2}a_{t_1,j})(a_{i,t_1}a_{t_2,i})]$ which is a permutation track of $Id$ where $t_1=t_2$.

In that same way, $\forall  (a_{i,t_1}a_{t_1,i})(a_{j,t_2}a_{t_2,j})$  in the permutation tracks of Id there exists another Id permutation track $(a_{t_1,i}a_{i,t_1})(a_{t_2,j}a_{j,t_2})$  where $t_1\ne t_2$ (notice that~$t_1=i$ and $t_2=j$ or vice versa were already been dealt in the monomials that originate from $\alpha _1$, so we receive an equal monomial in a \textbf{different} track), and $ (a_{i,t_1}a_{t_2,j})(a_{j,t_2}a_{t_1,i})$  which is a permutation track of $-Id$ where $t_1=t_2$.
\noindent Overall we obtain $f_{A^2}(x^2)\vDash f_A(x)^2$. 

\end{exa}

\begin{cor}
Every corner root of $f_{A^m}$ is an $m^{th}$ power of a corner root of $f_A$.
\end{cor}

\begin{proof}
 According to Theorem 3.6. $f_{A^m} \vDash (f_A)^m$, meaning $$f_{A^m}(x^m)=(f_A(x))^m+g(x^m),\ where\ g(x)\in G[x].$$ For every $\lambda ^m$ a corner root of $f_{A^m}$ we know that $$f_{A^m}(\lambda ^m)\in G.$$ 

\noindent If $$f_{A^m}(\lambda ^m)=(f_A(\lambda ))^m\in G\ then\ f_A(\lambda ) \in G$$ and therefore $\lambda$ is an eigenvalue of $A$. 

\noindent If $$f_{A^m}(\lambda ^m)=g(\lambda^m)\in G$$ then $\lambda ^m$ is not obtained as common value of two leading tangible monomials in $f_{A^m}$ but obtained as a value of a leading ghost monomial and therefore not a corner root of~$f_{A^m}$ in contradiction to the assumption in the Corollary.
\end{proof}

\end{document}